\documentclass[final,12pt,times,english,twoside,reqno]{amsart}
\usepackage{amssymb,mathrsfs}
\usepackage{amsmath,amsthm,amsfonts,latexsym,amssymb}
\usepackage{amscd,color,cite,epsfig,pgf,verbatim}
\usepackage{hyperref,enumitem,setspace}
\usepackage{mathtools}
\usepackage{geometry}
\theoremstyle{plain}
\newtheorem{theorem}[subsection]{Theorem}

\newtheorem{proposition}[subsection]{Proposition}

\theoremstyle{definition}
\newtheorem{remark}[subsection]{Remark}
\newtheorem{definition}[subsection]{Definition}
\newtheorem{example}[subsection]{Example}

\newcommand{\Addresses}{{
  \bigskip
  \footnotesize
    $^{1}$  University ''Alexandru Ioan Cuza'', Faculty of Mathematics,
  Bd. Carol I, No. 11, Ia\c{s}i, 700506, ROMANIA,
  email:   croitoru@uaic.ro, Orcid ID: 0000-0001-8180-3590
\\
    	$^{2}$  Petroleum-Gas University of Ploie\c{s}ti, Department of Computer Science,
      Information Technology, Mathematics and Physics,
         Bd. Bucure\c{s}ti, No. 39, Ploie\c{s}ti 100680, ROMANIA\\
          email: emilia.iosif@upg-ploiesti.ro, Orcid ID: 0000-0003-0144-8811
\\
$^{3}$   Department of Mathematics and Computer Sciences,
     University of  Perugia -- 1, Via Vanvitelli - 06123, Perugia, ITALY,
     email: anna.sambucini@unipg.it, Orcid ID 0000-0003-0161-8729 	
}}

\title[Generalized decomposition  integral ...]{Generalized decomposition  integral of real functions with respect  to fuzzy measures}
\author{ Anca Croitoru$^{1}$,   Alina Iosif$^{2}$,  and Anna Rita Sambucini$^{3}$}

\begin{document}
\date{\today}


\maketitle

\textbf{Abstract}. In this paper we define a  type of generalized Riemann-Lebesgue (decomposition) 
integral for non-negative real functions with respect to two non-additive set functions. For this  integral we present some classical properties.

\section{Introduction}

Different types of integrals in non-additive or set-valued frameworks have been developed, 
motivated by challenges in diverse fields such as economics, game theory, fuzzy logic, and data mining (see \cite{SC2025,Torra,prec,MS2024,retina,CS2024,BC,BR2009,CS2015,C2018,K2016,K2021,K2024,cgis2022,ccgis,cms,ccgism,Gal,LWY,ar,ZMP0,ZMP,LK2013,LLWY2025,AMH2025} and the references therein). \\
In the literature, various generalizations of the classical Riemann and Lebesgue integrals are known. 
One notable extension, called the Riemann-Lebesgue integral, was introduced by Kadets and Tseytlin in 2000 \cite{KT} 
for vector-valued functions with respect to countably additive measures. Comparative studies between 
the Birkhoff integral and the Riemann-Lebesgue integral have been presented in \cite{Po}.\\
 More recently, in 2020, Candeloro et al. \cite{ccgis} explored some properties of the Riemann-Lebesgue integral within the non-additive setting.

In this paper, we propose a new type of decomposition integral based on  Riemann-Lebesgue integrability 
for nonnegative real-valued functions with respect to two non-additive set functions. Our approach is inspired by previous works such as
 \cite{Gal,ZMP0, ZMP,AMH2025,MAL2025,EL2014,LMOS}.
 The paper is organized as follows: Section 2 introduces the topic and  reviews some fundamental concepts.
 In Section 3, we define the generalized decomposition integral for nonnegative functions relative to two non-additive set functions, 
 and we examine key properties such as monotonicity with respect to the set, the integrand, and the set functions; homogeneity; 
 additivity, with respect to the set and the set functions, and transformation rule. Finally concluding remarks are presented.
\\
The development of this generalized integral may open new avenues for applications across various fields: 
for example, in 
Economics, this framework can be used to model preferences or utility functions in uncertain environments
 where traditional additive measures are insufficient. 
 The integral can be also applied to analyze strategies in cooperative or non-cooperative games where pay-offs or
  utility functions are non-additive, allowing for more flexible modeling of coalition formation and strategic interactions 
 or it can provide a tool for aggregating fuzzy data, leading to more accurate reasoning in fuzzy inference systems.
\\
Finally, 
the generalized non-additive integral can be a valuable tool in image processing and analysis, 
especially when dealing with  noisy data. In image reconstruction, for example, pixel intensities
 or features often come with uncertainty that traditional additive measures may not adequately capture. 
 By using this non-additive integral, it becomes possible to model and aggregate information from different regions or sources. 
 This approach can improve the robustness of image reconstruction, enhance noise reduction, 
 and better preserve important details, leading to higher quality and more reliable images, see \cite{ccgism,CS2024,retina}.

\section{Preliminaries}
Denote $\mathbb{N}^{\ast }=\{1,2,3,...\}$. Let   $\mathbb{R}_{0}^{+}=[0, \infty)$ and  $(X, \| \cdot\|)$ be a Banach space.
Suppose $S$ is a nonempty set, at least countable,  and $\mathcal{C}$ a $\sigma$-algebra of subsets of $S$.
For every nonempty set $A \subset S$, let $\mathcal{P}(A)$ be the family of all
subsets of $A$. 
As usual, let $A^c=S\setminus A$ and let $\chi_{A}$ be the characteristic function of $A$. 
 If $P$ and $P^{\prime }$ are two countable 
partitions of $S$, then $P^{\prime }$ is said to be \textit{finer than} $P$,
\begin{eqnarray}\label{minore-ouguale}
\phantom{a} \,\,\,\,  P\leq P^{\prime } \,\, \mbox{ (or  } P^{\prime }\geq P \mbox{), if every set of }
P^{\prime } \mbox{   is included in some set of    } P.
\end{eqnarray}
We will use the symbol $\mathscr{P}(A)$ to denote the family of all countable partitions of $A$ 
whose elements belong to $\mathcal{C}$; if $A= S$ we will use simply $\mathscr{P}$.

For a set function $m :\mathcal{C}\rightarrow \mathbb{R}_{0}^{+}$,
the usual definitions 
 as in \cite{ccgis,ccgism,Den} are considered. 
For the sake of the completeness, we recall some of them.
Throughout the paper we consider 
 set functions $m$ such that $m(\emptyset)= 0.$
\begin{definition}  \cite{Den}
Consider a  set function $m :\mathcal{C}\rightarrow \mathbb{R}_{0}^{+}$.
Let $B, C$ be arbitrary sets in $\mathcal{C}$. Then $m$ is called:
\begin{itemize}
\item[(i)] monotone if $m(B) \leq m(C)$,  when $B\subseteq C$;
\item[(ii)] fuzzy if $m$ is monotone and $m(\emptyset)=0;$
\item[(iii)] submodular if $m(B\cup C)+ m(B\cap C)\leq m(B)+ m(C)$;
\item[(iv)] additive if $m(B\cup C)= m(B)+ m(C)$, when $B\cap C= \emptyset$;
\item[(v)] subadditive if $m(B\cup C)\leq m(B)+ m(C)$;
\item[(vi)] superadditive if $m(B\cup C)\geq m(B)+ m(C)$.
\end{itemize}
\end{definition}

\begin{definition}  \rm (\hskip-.07cm\cite[Definition 4]{ccgis}) 
Let $m :\mathcal{C}\rightarrow \mathbb{R}_{0}^{+}$ be a set function.
\begin{itemize}
\item[(i)] The variation of $m$ is the set function $\overline{m}:\mathcal{P}(S)\rightarrow \lbrack 0,+\infty ]$
defined,  for
every $B\subset S$, as 
 $\overline{m}(B)=\sup \{\sum\limits_{i=1}^{n}m(B_{i})\}$, where the supremum is extended over all finite
families of pairwise disjoint sets $\{B_{i}\}_{i=1}^{n}\subset \mathcal{C}$,
with $B_{i}\subseteq B$, for every $i\in \{1,\ldots ,n\}$. 
\item[(ii)] $m$ is said to be \textit{of finite variation on }$\mathcal{C}$
if $\overline{m}(S)<\infty $.
\end{itemize}
\end{definition}

A property $(P)$ holds $m$-almost everywhere (denoted by
$m$-a.e.) if there exists $E \in \mathcal{C}$, with $m(E) = 0$, so that the property $(P)$ is valid on $S\setminus E.$

\begin{definition}\label{def-int}
 \rm (\hskip-.07cm \cite[Definition 5]{ccgis})
 Let $\nu :\mathcal{C}\rightarrow \mathbb{R}_{0}^{+}$ be a set function.
A vector function $g:S\rightarrow X$ is called absolutely
(unconditionaly respectively) $\nu$-Riemann-Lebesgue integrable (on $S$) 
 if there exists $a\in X$
such that for every $\varepsilon>0$, there exists
$\Pi_{\varepsilon}\in \mathscr{P}$, such that for every $  \Pi \in \mathscr{P}$, $\Pi=(E_{n})_{n\in \mathbb{N}}$, finer than  
$\Pi_{\varepsilon}$, ($\Pi \geq \Pi_{\varepsilon}$ in the sense of \eqref{minore-ouguale})
\begin{itemize}
\item $g$ is bounded on every $E_{n}$, with $\nu(E_{n})>0$ and
\item for every $s_{n}\in E_{n}$, $n\in \mathbb{N}$,
the series $\sum\limits_{n=0}^{+\infty}g(s_{n})\nu(E_{n})$ is absolutely
(unconditionaly respectively) convergent and
$$\left\|\sum_{n=0}^{+\infty}g(s_{n})\nu(E_{n})-a \right\|<\varepsilon.$$
\end{itemize}
\end{definition}
\begin{remark} \rm
We call $a={\scriptstyle ((A)RL)}\int_S gd\nu$ the Riemann-Lebesgue integral of $g$ (on $S$) with respect to $\nu$.
This integral was introduced first in \cite{KT, Po} in the countably additive case. 
Obviously if $a$ exists, then it is unique. According to  \cite{ccgis},
the sets of all absolutely
(unconditionally respectively) Riemann-Lebesgue integrable functions on $S$,  are linear spaces.
 If $X$ is finite dimensional, the two classes coincide and we denote 
 with $RL(\nu, S)$ this class.
\end{remark}
Other properties of Riemann-Lebesgue integrable functions in non-additive case can be found in \cite{cgis2022,ccgis,ccgism, c2025}.
Note that, thanks to the definition, the  introduced integral which is a decomposition-type integral, 
does not need the measurability of integrands. It is also additive, while integrals such as Choquet, Pan, 
concave or Shilkret are generally only subadditives.
On the other hand not all characteristic functions are integrable in this sense and, those that are, 
do not have in general integral equal to the $\nu$ value of the set.\\
The  described definition  permits the integration of functions with respect to measures that are not 
necessarily additive, such as capacities or, more generally, fuzzy measures. 

\section{A Generalized Decomposition integral}
Originally,
as highlighted in the Introduction,
 the integrals with respect to non additive measures were applied in potential theory and statistical mechanics, 
 and they  have evolved into a valuable tool for addressing uncertainty within the frameworks of imprecise probability theory,
  decision theory, and the analysis of cooperative games with  applications extend to fields such as finance, economics, and insurance.\\
In order to motivate the generalized decomposition integral that will be considered in this section we recall, as an example, that 
a fundamental challenge in Mathematical Economics involves identifying equilibrium. In 
\cite{cms}
 a model was considered where the space of agents is partitioned into a large number of sections, 
 each representing an autonomous economic subsystem. Additionally, coalitions may form across 
 members of different sections according to specified rules.\\
  The mathematical framework employed was a product space  $X^* := X \times [0,1]$, 
  where each section corresponds to the set $ X \times \{ y\}$. $X$  denotes a typical section of agents
   and is equipped with a $\sigma$-algebra, while the interval $[0,1]$ was endowed with the 
   standard Lebesgue $\sigma$-algebra  $\mathcal{B}$  and the Lebesgue measure $\lambda$.
Within each section  the $\sigma$-algebra product was considered, along with a fuzzy measure  $\mu_y$  defined on it. \\

Following this idea we  introduced here an integral, based on the Riemann-Lebesgue integrability, 
 for functions defined in 
$S \times \mathbb{R}_0^+$. In this setting also $\mathbb{R}_0^+$ will be
 associated to a set-function not necessarily additive. Results of this type were also given in 
\cite{Gal, ZMP0, ZMP,AMH2025,MAL2025,LLWY2025}
 for the Choquet integral. \\

\begin{itemize}
\item[(H)]
So, let $(S, \mathcal{C})$, and $(\mathbb{R}_0^+, \mathcal{E})$ be two measurable spaces endowed with two $\sigma$-algebras and two 
set-functions
 $\mu:\mathcal{C}\rightarrow \mathbb{R}_0^+$
 and  $\nu:\mathcal{E}\rightarrow \mathbb{R}_0^+$,
 both vanishing on the empty set and such that $\{0\}\in \mathcal{E}$.\\
 \end{itemize}

 Let $\mathscr{F}(\mu)$ the family of all $\mu$-measurable functions $f: S \to \mathbb{R}_0^+$. \\
For every  $f \in \mathscr{F}(\mu)$
 and for every 
 $A\in \mathcal{C}$,   let
 $E_f^{\alpha}:= \{s \in S: f(s) \geq \alpha\}$ and  
  $u_{f, \mu}^{A}:\mathbb{R}_0^+\rightarrow \mathbb{R}_0^+$ be the function 
defined  by 
\begin{eqnarray}
\label{u-f}
u_{f, \mu}^{A}(\alpha)=\mu ( \{s\in A: \,  f(s)\geq \alpha\})
=\mu(E_f^{\alpha} \cap A), 
 \qquad \forall \, \alpha \geq 0.
\end{eqnarray}

\begin{definition}\label{def-CRL}
We say that a $\mu$-measurable function $f:S\rightarrow \mathbb{R}^+_0$ is $(\nu, \mu)$-integrable on 
$A\in \mathcal{C}$ if $u_{f, \mu}^{A}$ is $\nu$-RL integrable on $\mathbb{R}_0^+$.
In this case
\begin{eqnarray}\label{g-int}
\int^{*}_{A}f \, d(\nu, \mu) := 
{{\scriptstyle (RL)}}\int_{[0, \infty)} u_{f, \mu}^{A}\, d\nu =
{\scriptstyle(RL)}\int_{[0, \infty)} \mu(\{s \in A : f(s) \geq \alpha \}) \, d\nu(\alpha)
\end{eqnarray}
 is called the generalized decomposition integral of $f$ on $A$ with respect to $(\nu, \mu)$.
 \end{definition}
For every  $A \in \mathcal{C}$ let $G_{RL}(\nu, \mu, A)$ be the space of all $\mu$-measurable non-negative functions $f$
 such that the integral \eqref{g-int}, valued on $A$, is finite.
If $A=S$ we denote it with $G_{RL}(\nu, \mu)$.

\begin{remark}
\rm
If $\nu$ is the Lebesgue measure, then the above integral reduces to  the Choquet integral.
\end{remark}

\begin{example} \rm \phantom{a}
\begin{itemize}
\item If $\mu(A)=0$ for every $A\in \mathcal{C}$ and
 $\nu$ is of finite variation,   then 
 $\displaystyle{\int^{*}_{S}} f \, d(\nu, \mu)= 0.$
 \item Suppose that $\nu$ is RL-integrable, according to \cite[Definition 7]{cgis2022}.
 If $f$ is a constant function, $f(s) = c \in \mathbb{R}_0^+$, for every $s \in S$, then
 $f \in G_{RL}(\nu, \mu)$ and
 $\displaystyle{\int^{*}_{S}}f \, d(\nu, \mu)= \mu(S) \cdot \nu([0,c]).$
 \item
 Let $(\mathbb{R}_0^+, \mathcal{B}, \lambda) $ with 
$\mathcal{B}$ the Borel $\sigma$-algebra and $\lambda$ be the Lebesgue measure.
Let  $(S,\mathcal{C}, \mu) = ([0,1], \mathcal{B}([0,1]), \lambda^2)$.
 Let $f:[0,1]\to \mathbb{R}$ be defined by $f(x) = x$; $f$ is  measurable.
 In this case, for every $\alpha \geq 0$ and $A \in
\mathcal{B}([0,1])$,  
   it is 
\[
 u_{f, \mu}^{A}(\alpha)=\mu ( \{x\in A: \,  x\geq \alpha\})
=\lambda^2(\{x \in [0,1] : x \geq \alpha\}  \cap A) .
\]
 Now, let $ A = [0,1]$.
 Then 
$ u_{f, \mu}^{[0,1]}(\alpha)= (1-\alpha)^2 \chi_{[0,1]} $.
  In this case the integral coincides with the Riemann one and then  
$\displaystyle{\int^{*}_{[0,1]}} f \, d(\nu, \mu)= \int_0^1 (1-\alpha)^2 d\alpha = \frac{1}{3}.$ 
\item 
Let  $(S,\mathcal{C}, \mu) = (\mathbb{N}, \mathcal{P}(\mathbb{N}), \mu)$ with
 $\mu(\emptyset) =0$ and $\mu(A) =1$ otherwise.
 Let $f: \mathbb{N} \to \mathbb{R}$ defined by $f(n) = n$, as before $f$ is  measurable.\\
Let $(\mathbb{R}_0^+, \mathcal{B}, \nu) $ with 
$\mathcal{B}$ the Borel $\sigma$-algebra and 
$\nu(E) = 0$ if $E$ is a bounded set and $\nu(E) =1$ if
$E$ is unbounded.  
 In this case, for every $\alpha \geq 0$ and $A \in
\mathcal{P}(\mathbb{N})$,  
   it is 
\[
 u_{f, \mu}^{A}(\alpha)=\mu ( \{n\in A: \,  f(n)\geq \alpha\})
=\mu(\{n \in \mathbb{N} : n \geq \alpha\}  \cap A) = 1. 
\]
 Now, let $ A = \mathbb{N}$,  fix $\varepsilon > 0$ and consider $P_{\varepsilon}$ a countable partition of $\mathbb{R}_0^+$
 composed by bounded sets, for example $P_{\varepsilon} := \{ [n,n+1[, \, n \in \mathbb{N}\}$.
 Then for every partition $P =\{E_n, n \in \mathbb{N}\}  \geq P_{\varepsilon}$ we have
\[\sum_{n=0}^{+\infty} u_{f, \mu}^{\mathbb{N}}(\alpha)\nu(E_{n}) =\sum_{n=0}^{+\infty} \nu(E_{n}) = 0\]
 and then
$\displaystyle{\int^{*}_{\mathbb{N}}} f \, d(\nu, \mu)= 0.$
\end{itemize}
\end{example}

Troughout  the section we refere to $\mu$ defined on the measurable space
 $(S,\mathcal{C})$ and $\nu$ to $(\mathbb{R}_0^+, \mathcal{E}) $ and both satisfies conditions given in (H). 
 
\subsection*{A - Case:  $\nu$  of finite variation }
In this subsection properties of the $G_{RL} (\nu, \mu)$ integral are exposed when $\nu$ is of finite variation.
We start with a dominating result.
\begin{proposition}\label{P1}
Let    $f \in \mathscr{F}(\mu)$ be 
such that $u_{f, \mu}^{S}$ is bounded. 
Then $f \in G_{RL} (\nu, \mu)$-integrable and
\[
  \int^{*}_{S}f\, d(\nu, \mu)\leq \overline{\nu}(\mathbb{R}_0^+)\cdot \sup_{\alpha \geq 0}\, u_{f, \mu}^{S}(\alpha).
  \]
\end{proposition}
\begin{proof}
Bounded and measurable functions $u_{f, \mu}^{S}$ are RL-integrable with respect to $\nu$ and 
the assertion follows by \cite[Proposition 1]{ccgis}.
\end{proof}
\begin{remark}\rm
From Proposition \ref{P1} it follows that 
$\mathscr{F}(\mu) \subset G_{RL} (\nu, \mu)$ 
when $\mu$ is fuzzy and $\nu$ is of finite variation, since
$\sup_{\alpha \in [0, \infty)}u_{f, \mu}^{S}(\alpha) \leq \mu(S) < +\infty.$
\end{remark}

The integrability of a function in a measurable set is obtained via the
integrability on the whole set $S$  of the product of the function with 
  the characteristic function of the measurable set by  Theorem \ref{Te}. \\

In the subsequent Theorems \ref{Te}, \ref{Te2} and \ref{0.a.e.} of this subsection we suppose also that  $\nu(\{0\})=0$.
\begin{theorem}\label{Te}
 Let 
 $ A\in \mathcal{C}$ be fixed and let
 $f \in \mathscr{F}(\mu) \cap G_{RL}(\nu, \mu, A)$.
  Then
\begin{equation*}
  \int^{*}_{A} f \, d(\nu, \mu)= \int^{*}_{S}f \chi_{A}\, d(\nu, \mu).
\end{equation*}
\end{theorem}
\begin{proof}
By definition
\begin{equation*}
  \int^{*}_{A}f\, d(\nu, \mu)= 
  {\scriptstyle (RL)}\int_{S}u_{f, \mu}^{A} (\alpha)\, d\nu (\alpha),
\end{equation*}
 where $u_{f, \mu}^{A}(\alpha)= \mu(E_{f}^{\alpha}\cap A)$ for every  $\alpha\geq 0$ and
 \begin{equation*}
 \int^{*}_{S}f\chi_{A}\, d(\nu, \mu)= {\scriptstyle (RL)}\int_{S}u_{f\chi_{A}, \mu}^{S}(\alpha) \, d\nu(\alpha),
 \end{equation*}
 where
\[
u_{f\chi_{A}, \mu}^{S} (\alpha)=
\begin{cases} 
\mu(S) & \text{ if } \alpha=0 \\
 \mu(E_{f}^{\alpha}\cap A)& \text{ if }  \alpha>0 .
 \end{cases} 
 \]
 Let us observe that 
 $u_{f, \mu}^{A}(\alpha)= u_{f\chi_{A}, \mu}^{S}(\alpha)$
 \, $\nu$-a.e. and
\[
\sup_{\alpha\in[0, \infty)}| u_{f, \mu}^{A}(\alpha)- u_{f\chi_{A}, \mu}^{S}(\alpha)|<+\infty.
\]
 So we apply \cite[Corollary 2]{ccgis} and  obtain
 \[
 \int^{*}_{A}f\, d(\nu, \mu):=
 {\scriptstyle (RL)}\int_{S}u_{f, \mu}^{A}(\alpha) \, d\nu(\alpha)=
  {\scriptstyle (RL)}\int_{S}u_{f\chi_{A}, \mu}^{S}(\alpha)\, d\nu(\alpha)
  := \int^{*}_{S}f \chi_{A}\, d(\nu, \mu),
  \]
  which concludes the proof.
\end{proof}
 Moreover the integrability is hereditary on the measurable subsets, as  Theorem \ref{Te2} shows.
\begin{theorem}\label{Te2}
Suppose   $\mu:\mathcal{C}\rightarrow \mathbb{R}_0^+$ is bounded. If $f\in
\mathscr{F}(\mu) \cap  G_{RL}(\nu, \mu)$ then, for every $A\in \mathcal{C}$, 
$f\in G_{RL}(\nu, \mu, A)$ and
\begin{equation*}
  \int^{*}_{A}fd(\nu, \mu)= \int^{*}_{S}f\chi_{A}d(\nu, \mu).
\end{equation*}
\end{theorem}
\begin{proof}
Let $A\in \mathcal{C}$. 
According to Definition \ref{def-CRL}, we have to prove that the function 
$u_{f, \mu}^{A}(\alpha)= \mu(E_{f}^{\alpha}\cap A)$
 is Riemann-Lebesgue integrable with respect to $\nu$. 
 Since $\mu$ is bounded, then the function
$u_{f, \mu}^{A}$ is bounded too. 
According to \cite[Proposition 1]{ccgis} it results that 
$u_{f, \mu}^{A}\in RL(\nu, S).$ 
Now the conclusion follows by Theorem \ref{Te}.
\end{proof}

Finally the integral of a function which is $\mu$-a.e. zero valued is null.
\begin{theorem}\label{0.a.e.}
Let  $\mu:\mathcal{C}\rightarrow \mathbb{R}_0^+$ be a fuzzy measure.
 Let $f \in \mathscr{F}(\mu)$ be  such that 
  $f=0$ $\mu$-a.e.. Then $f\in G_{RL}(\nu,\mu)$
and 
\[\int_{S}^{*}f\, d(\nu, \mu)=0.\]
\end{theorem}
\begin{proof}
  Suppose $f=0$ $\mu$-a.e. and consider $B= \{s\in S; f(s)>0\}$. Then $\mu(B)=0$ and $f(s)=0$,  for every $s\in S\setminus B.$
  Let $\alpha \geq 0.$ \\
  If $\alpha=0$, we have
   $E_{f}^{0}=\{s\in S; f(s) {\color{blue} \geq } 0\}= S$
    and for every $\alpha>0$, the set 
  $E_{f}^{\alpha}=\{ s\in S; f(s)\geq \alpha\} \in \mathcal{C} \cap B$. 
  Since $\mu$ is monotone it results that $\mu(E_{f}^{\alpha})=0.$ Therefore,
\[
 u_{f, \mu}^{S}(\alpha)= 
\begin{cases} 
\mu(S) & \text{ for } \alpha=0 \\ 0 
& \text{ for } \alpha>0. 
\end{cases} 
\]
Since $\nu(\{0\})=0$, then 
$u_{f, \mu}^{S}(\alpha)=0$ $\nu$-a.e.  and it is bounded by $\mu(S)$. \\
By \cite[Theorem 2]{ccgis}, 
$u_{f, \mu}^{S}\in RL(\nu, S)$ and 
\[ 0= {\scriptstyle (RL)}\int_{S}u_{f, \mu}^{S}(\alpha)\, d\nu(\alpha)= 
\int_{S}^{*}f \, d(\nu, \mu).\]
\end{proof}
\subsection*{B - Case: additional assumptions on $\mu$}
In this subsection we will introduce additional assumptions on the set function $\mu$ in order
 to obtain some other  properties of the $G_{RL}$ integral.
\\
Independently from Theorem \ref{0.a.e.} we can obtain a result of integrability when integrands are 
equal $\mu$-a.e. imposing assumptions on $\mu$ rather than $\nu$.
\begin{proposition}
Suppose $\mu:\mathcal{C}\rightarrow \mathbb{R}_0^+$ is a fuzzy subadditive measure.
 Let $f, g \in \mathscr{F}(\mu)$ be
such that $f=g$ $\mu$-a.e.. 
If $f\in G_{RL}(\nu,\mu)$, then $g\in G_{RL}(\nu,\mu)$ and 
\[\int_{S}^{*}f\, d(\nu, \mu)= \int_{S}^{*}g\, d(\nu, \mu).\]
\end{proposition}
\begin{proof}
Let $B= \{s\in S; f(s)\neq g(s)\}$. 
Then $\mu(B)=0$ and $f(s)= g(s)$ for every  $s\in S\setminus B.$
We observe that for every $\alpha \geq 0$, 
\[ E_{f}^{\alpha}\subseteq E_{g}^{\alpha}\cup B, \qquad  E_{g}^{\alpha}\subseteq E_{f}^{\alpha}\cup B,\]
Since $\mu$ is monotone and subadditive, it follows that
for every  $ \alpha \geq 0$ it is
$\mu(E_{f}^{\alpha})= \mu(E_{g}^{\alpha})$ and 
this implies that 
$u_{f, \mu}^{S}(\alpha)= u_{g, \mu}^{S}(\alpha)$ 
and the conclusion holds.
\end{proof}
\begin{proposition}\label{Tmon}
Let $\mu:\mathcal{C}\rightarrow \mathbb{R}_0^+$ be a fuzzy measure and 
$f \in \mathcal{F}(\mu)$. If $f $  is $G_{RL} (\nu,\mu)$ integrable  on the sets
$A, B\in \mathcal{C}$, with $A\subseteq B$, then
\begin{equation*}
  \int_{A}^{*}f\, d(\nu, \mu)\leq \int_{B}^{*}f\, d(\nu, \mu).
\end{equation*}
\end{proposition}
\begin{proof}
Since $A\subseteq B$, we have $E_{f}^{\alpha}\cap A\subseteq E_{f}^{\alpha}\cap B$, for every $\alpha \geq 0.$\\
By the monotonicity of 
$\mu$  it follows that $\mu(E_{f}^{\alpha}\cap A)\leq \mu(E_{f}^{\alpha}\cap B)$ for every $\alpha \geq 0.$
Applying \cite[Theorem 6]{ccgis}, it results
\begin{align*}
  \int_{A}^{*}f\, d(\nu, \mu)=
   {\scriptstyle (RL)}\int_{S}u_{f,\mu}^{A}\, d\nu\leq 
  {\scriptstyle (RL)}\int_{S}u_{f, \mu}^{B}\, d\nu=
   \int_{B}^{*}f\, d(\nu, \mu).
\end{align*}
\end{proof}
The monotonicity of the set-function $\mu$ allows also to obtain monotonicity results between  integrals and 
  integrands (Theorem \ref{Tmf}) and between integrals and set-functions  (Theorems \ref{Tmm} and \ref{Tmn}).
\begin{theorem} \label{Tmf}
Let $\mu:\mathcal{C}\rightarrow \mathbb{R}_0^+$ be a fuzzy measure and
 $f_{1}, f_{2}\in \mathscr{F} (\mu)$ such that $f_{1}\leq f_{2}$. 
 If $f_{1}, f_{2} \in  G_{RL}(\nu,\mu, A)$, for some  $A\in \mathcal{C}$, then
\begin{equation*}\label{1}
  \int_{A}^{*}f_{1}\, d(\nu, \mu)\leq 
  \int_{A}^{*}f_{2}\, d(\nu, \mu).
\end{equation*}
\end{theorem}
\begin{proof}
We observe that $E_{f_{1}}^{\alpha}\subseteq E_{f_{2}}^{\alpha}$, for every $\alpha\geq 0.$ 
Then, since $\mu$ is monotone, we have $\mu(E_{f_{1}}^{\alpha}\cap A)\leq \mu(E_{f_{2}}^{\alpha}\cap A)$ for every $ \alpha\geq 0.$ \\
According to \cite[Theorem 6]{ccgis}, it follows
\begin{equation*}
  \int_{A}^{*}f_{1}\, d(\nu, \mu)=
   {\scriptstyle (RL)}\int_{S}\mu(E_{f_{1}}^{\alpha}\cap A)\, d\nu\leq 
    {\scriptstyle (RL)}\int_{S}\mu(E_{f_{2}}^{\alpha}\cap A)\, d\nu= 
   \int_{A}^{*}f_{2}\, d(\nu, \mu).
\end{equation*}
\end{proof}

\begin{remark} Theorem \ref{Tmf} also works in the following hyphotesis: $\mu$ is a complete finitely additive measure and $f_{1}\leq f_{2}$ $\mu$-ae.
\end{remark}
\begin{theorem}\label{Tmm}
Let $\mu_{1}, \mu_{2}:\mathcal{C}\rightarrow \mathbb{R}_0^+$ 
 be such that 
$\mu_{1}\leq \mu_{2}$ setwise, (namely 
 $\mu_{1}(E)\leq \mu_{2}(E),$
  for every $ E\in \mathcal{C}$) and a function 
  $f \in \mathscr{F} (\mu_i)$, $i=1,2$.
  Let $A\in \mathcal{C}$ be such that 
$f\in G_{RL}(\nu,\mu_{i}, A)$, for  $i=1,2$, 
then
\begin{equation*}\label{1}
  \int_{A}^{*}fd(\nu, \mu_{1})\leq \int_{A}^{*}fd(\nu, \mu_{2}).
\end{equation*}
\begin{proof}
By hypothesis, for every $\alpha\geq 0$ we have
$\mu_{1}(E_{f}^{\alpha}\cap A)\leq \mu_{2}(E_{f}^{\alpha}\cap A).$\\ By \cite[Theorem 7]{ccgis}, it follows
\begin{eqnarray*}
  \int_{A}^{*}fd(\nu, \mu_{1}) &=&
   {\scriptstyle (RL)} \int_{S}\mu_{1}(E_{f}^{\alpha}\cap A) \, d\nu
   \leq \\
   &\leq& 
  {\scriptstyle (RL)}\int_{S}\mu_{2}(E_{f}^{\alpha}\cap A)d\nu=
    \int_{A}^{*}fd(\nu, \mu_{2}),
\end{eqnarray*}
which finishes the proof.
\end{proof}
\end{theorem}
\begin{theorem}\label{Tmn}
Let $\nu_{i}, :\mathcal{E}\rightarrow \mathbb{R}_0^+$, $i=1,2$
 be such that 
 $\nu_{1}\leq \nu_{2}$ setwise in $ \mathcal{E}$ and 
 $f \in \mathscr{F}(\mu)$. 
 If there is $A\in \mathcal{C}$ such that
$f\in G_{RL}(\nu_{1},\mu, A)\cap G_{RL}(\nu_{2},\mu, A)$,
then
\begin{equation*}
  \int_{A}^{*}f \, d(\nu_{1},\mu)
  \leq \int_{A}^{*}f\, d(\nu_{2},\mu).
\end{equation*}
\end{theorem}
\begin{proof}
Since $f\in G_{RL}(\nu_{1},\mu, A)\cap G_{RL}(\nu_{2},\mu, A)$,
 then by \cite[Theorem 7]{ccgis}, it results
\begin{eqnarray*}
  \int_{A}^{*}f\, d(\nu_{1},\mu) &:=&
   {\scriptstyle (RL)}\int_{S}\mu(E_{f}^{\alpha}\cap A)\, d\nu_{1} 
   \leq \\
   &\leq&
  {\scriptstyle (RL)}\int_{S}\mu(E_{f}^{\alpha}\cap A)d\nu_{2} :=
   \int_{A}^{*}f\, d(\nu_{2},\mu).
\end{eqnarray*}
\end{proof}
\subsection*{C - The additivity with respect to the set-functions}
In this subsection we will prove that the integral is additive with respect to set-functions, proving it separately.
\begin{proposition}
Let $\mu,\nu$  be as in (H) and $f \in \mathscr{F}(\mu)$. 
If there is  $A\in \mathcal{C}$  such that $f\in G_{RL}(\nu,\mu, A)$, then, for every $a, b > 0$, $f\in G_{RL}(a\nu, b\mu, A)$ and
\begin{equation*}\label{1}
  \int_{A}^{*}f\, d(a\nu, b\mu)= ab \cdot \int_{A}^{*}fd(\nu, \mu).
\end{equation*}
\end{proposition}
\begin{proof}
The conclusion follows from \cite[Theorem 3]{ccgis} and we have
\begin{align*}
  \int_{A}^{*}f \, d(a\nu, b\mu)= 
  {\scriptstyle (RL)}\int_{[0, \infty)}b\mu(E_{f}^{\alpha}\cap A)\, d(a\nu)=\\
  =ab \cdot {\scriptstyle (RL)}\int_{[0, \infty)}\mu(E_{f}^{\alpha}\cap A)\, d\nu= 
  ab \cdot \int_{A}^{*}fd(\nu, \mu).
\end{align*}
\end{proof}
\begin{proposition}\label{T8}
Let $\mu_{i}, \, i=1,2$ be set functions as in (H) and 
$ f \in \mathscr{F}(\mu_i), \, i=1,2$
If there is  $A\in \mathcal{C}$ such that
 $f\in G_{RL}(\nu,\mu_{i}, A)$, $i=1,2$, then $f\in G_{RL}(\nu,\mu_{1}+\mu_{2}, A)$ and
\begin{equation*}
  \int_{A}^{*}f \, d(\nu, \mu_{1}+\mu_{2})= 
  \int_{A}^{*}f \, d(\nu, \mu_{1})+ 
  \int_{A}^{*}f \, d(\nu, \mu_{2}).
\end{equation*}
\end{proposition}
\begin{proof}
The $ G_{RL}$ integrability of $f$ with respect to 
$(\nu,\mu_{1}+\mu_{2})$ on $A$ follows from 
\cite[Theorem 4]{ccgis}; moreover 
\begin{eqnarray*}
\int_{A}^{*}f\, d(\nu, \mu_{1}+\mu_{2}) &=&
{\scriptstyle (RL)}\int_{[0,\infty)}(\mu_{1}+\mu_{2})(E_{f}^{\alpha}\cap A)\, d\nu=\\
&=&{\scriptstyle (RL)}\int_{[0,\infty)}\mu_{1}(E_{f}^{\alpha}\cap A)\, d\nu+
{\scriptstyle (RL)}\int_{[0,\infty)}\mu_{2}(E_{f}^{\alpha}\cap A)\, d\nu=\\
 &=&\int_{A}^{*}f\, d(\nu, \mu_{1})+
 \int_{A}^{*}f\, d(\nu, \mu_{2}).
\end{eqnarray*}
\end{proof}
\begin{proposition}
Let $\nu_{i}, i=1,2,$ be set functions as in (H) and $f \in \mathscr{F}(\mu)$.  Let $A\in \mathcal{C}$  be such that 
 $f\in  G_{RL}(\nu_{i},\mu, A), \, i=1,2$, then 
 $f\in G_{RL}(\nu_{1}+\nu_{2},\mu, A)$ and
\begin{equation*}
  \int_{A}^{*}f\, d(\nu_{1}+\nu_{2},\mu)= 
  \int_{A}^{*}f\, d(\nu_{1},\mu)+ 
  \int_{A}^{*}f\, d(\nu_{2},\mu).
\end{equation*}
\end{proposition}
\begin{proof}
The conclusion holds by \cite[Theorem 4]{ccgis}  and
\begin{eqnarray*}
\int_{A}^{*}f\, d(\nu_{1}+\nu_{2},\mu) &=&
{\scriptstyle (RL)}\int_{[0, \infty)}\mu(E_{f}^{\alpha}\cap A)d(\nu_{1}+\nu_{2})=\\&=&
{\scriptstyle (RL)}\int_{[0, \infty)}\mu(E_{f}^{\alpha}\cap A)d\nu_{1}+
{\scriptstyle (RL)}\int_{[0, \infty)}\mu(E_{f}^{\alpha}\cap A)d\nu_{2}=\\
&=&\int_{A}^{*}f\, d(\nu_{1},\mu) +
\int_{A}^{*}f\, d(\nu_{2},\mu).
\end{eqnarray*}
\end{proof}
Finally, using again monotonicity of $\mu$ we are able to consider the integrability of the supremum
 or the infimum of two integrands asking, a priory, the integrability of all the involved functions.
\begin{theorem}
Let $\mu$ be a fuzzy submodular measure and  $f, g, f\vee g, f\wedge g \in \mathscr{F}(\mu) \cap G_{RL} (\nu,\mu, A)$ for some 
$A\in \mathcal{C}$,
then
\begin{equation*}
  \int_{A}^{*}(f\vee g)\, d(\nu, \mu)+
   \int_{A}^{*}(f\wedge g) \, d(\nu, \mu)\leq 
   \int_{A}^{*}f\, d(\nu, \mu)+
   \int_{A}^{*}g\, d(\nu, \mu).
\end{equation*}
\end{theorem}
\begin{proof}
We observe that $E_{f\vee g}^{\alpha}= E_{f}^{\alpha}\cup E_{g}^{\alpha}$
and $E_{f\wedge g}^{\alpha}\subset E_{f}^{\alpha}\cap E_{g}^{\alpha}$.
According to \cite[Theorems 3 and 6]{ccgis}, we have
\begin{eqnarray*}
&&\int_{A}^{*}(f\vee g)\, d(\nu, \mu)+ 
\int_{A}^{*}(f\wedge g)\, d(\nu, \mu)=\\ 
&&={\scriptstyle (RL)}\int_{[0, \infty)}\mu(E_{f\vee g}^{\alpha}
\cap A) \, d\nu+ 
{\scriptstyle (RL)}\int_{[0, \infty)}\mu(E_{f\wedge g}\cap A)\, d\nu\leq \\
&&\leq {\scriptstyle (RL)}\int_{[0, \infty)}\mu( (E_{f}^{\alpha}\cup E_{g}^{\alpha})\cap A)\, d\nu+
{\scriptstyle (RL)}\int_{[0, \infty)}\mu(E_{f}^{\alpha}\cap E_{g}^{\alpha}\cap A)\, d\nu \leq \\
&& \leq {\scriptstyle (RL)}\int_{[0, \infty)}\mu( E_{f}^{\alpha}\cap A)\, d\nu+ 
{\scriptstyle (RL)}\int_{[0, \infty)}\mu(E_{g}^{\alpha}\cap A)\, d\nu=
\\&& = \int_{A}^{*}f\, d(\nu,\mu)+
\int_{A}^{*}g \, d(\nu,\mu).
\end{eqnarray*}
\end{proof}
For what concernes the additivity with respect to the sets where we integrate, 
only a partial result is obtained for additive measures.
\begin{theorem}\label{T11}
Let $\mu:\mathcal{C}\rightarrow \mathbb{R}_0^+$ be a finitely additive measure, 
$A, B\in\mathcal{C}$ with  $A\cap B=\emptyset$ and $f \in \mathscr{F}(\mu)$.
 If $f\in  G_{RL}(\nu,\mu, A)\cap  G_{RL}(\nu,\mu, B)\cap  G_{RL}(\nu,\mu,  A\cup B)$, then
\begin{equation*} \label{1}
  \int_{A\cup B}^{*}f\, d(\nu, \mu)= 
  \int_{A}^{*}f\, d(\nu, \mu)+ 
  \int_{B}^{*}f\, d(\nu, \mu).
\end{equation*}
\end{theorem}
\begin{proof}
Applying the additivity of $\mu$ and \cite[Theorem 3]{ccgis}, we get
\begin{eqnarray*}
  \int_{A\cup B}^{*}f\, d(\nu, \mu) &=& 
{\scriptstyle (RL)}\int_{[0,\infty)}\mu(E_{f}^{\alpha}\cap (A\cup B))\, d\nu=
  {\scriptstyle (RL)}\int_{[0,\infty)}\mu(E_{f}^{\alpha}\cap A)\, d\nu+\\
  &+&{\scriptstyle (RL)}\int_{[0,\infty)}\mu(E_{f}^{\alpha}\cap B)\, d\nu=
  \int_{A}^{*}fd(\nu, \mu)+\int_{B}^{*}f\, d(\nu, \mu).
\end{eqnarray*}
\end{proof}

While for integrands, in general the integral of a sum is not the sum of integrals. 
As an example we can consider two scalar functions that are not comonotonic, 
$\nu=\lambda$  is the Lebesgue measure and we consider the Choquet integral.
\begin{theorem}\label{T12}
Suppose $\mu:\mathcal{C}\rightarrow \mathbb{R}_0^+$ is a superadditive fuzzy measure.
Let $f, g \in \mathscr{F}(\mu)$ be  such that
$f, g, f+g \in G_{RL}(\nu,\mu)$. Then
\begin{equation*} \label{1}
  \int_{S}^{*}(f+g)d(\nu, \mu)\geq 
  \int_{S}^{*}fd(\nu, \mu)+ \int_{S}^{*}gd(\nu, \mu).
\end{equation*}
\end{theorem}
\begin{proof}
  For every $\alpha\geq 0$, 
  $E_{f}^{\alpha}\cup E_{g}^{\alpha}\subset E_{f+g}^{\alpha}.$ 
  Since $\mu$ is monotone, then 
  $\mu(E_{f}^{\alpha}\cup E_{g}^{\alpha})\leq 
  \mu(E_{f+g}^{\alpha})$, for every $ \alpha> 0.$ 
  Applying superadditivity of $\mu$, we have
  \begin{equation*}
  u_{f+g, \mu}^{S}(\alpha)= 
  \mu(E_{f+g}^{\alpha})\geq \mu(E_{f}^{\alpha})+ \mu(E_{g}^{\alpha})=
   u_{f, \mu}^{S}(\alpha)+ u_{g, \mu}^{S}(\alpha),
  \end{equation*}
  for every $\alpha\geq 0$.
Now, the conclusion follows by   \cite[Theorems 6 and 3-3.c)]{ccgis}.
\end{proof}
Finally we consider an integration by substitution.
Let $T \neq \emptyset$, $\varphi: S\rightarrow T$ be a function and   
$\mathcal{A}=\{E\subset T; \varphi^{-1}(E)\in \mathcal{C}\}$.
Let $\mu\varphi^{-1}$ be  the set-function 
\[ \mu\varphi^{-1}:\mathcal{A}\rightarrow \mathbb{R}_0^+,
\,\, \mbox{ defined for every  } 
E\in \mathcal{A}\,\,
\mbox{   by   } \,\,
(\mu\varphi^{-1})(E)= \mu(\varphi^{-1}(E)).\]
It is known that $\mathcal{A}$ is a $\sigma$-algebra of subsets of $T$ and $(\mu\varphi^{-1})(\emptyset)= 0.$

\begin{theorem} (Transformation Rule)
  Let $T$ be a nonvoid set, $\varphi: S\rightarrow T$ a function,
   let $\mathcal{A}$ and $\mu\varphi^{-1}$ be defined as above and
consider a function $g:T\rightarrow \mathbb{R}_0^+$.
Then $g\in  G_{RL}(\nu,\mu\varphi^{-1}, T)$ if and only if 
$g\circ \varphi\in  G_{RL}(\nu,\mu, S)$.
 In this case,
 \begin{equation*}\label{2}
  \int_{T}^{*}g\, d(\nu, \mu\varphi^{-1})= \int_{S}^{*}(g\circ \varphi)\, d(\nu, \mu).
\end{equation*}
\end{theorem}
\begin{proof}
For every $\alpha \geq 0$, it holds
$\varphi^{-1}(E_{g}^{\alpha})= E_{g\circ \varphi}^{\alpha}$, which implies
 $\mu\varphi^{-1}(E_{g}^{\alpha})= \mu(E_{g\circ \varphi}^{\alpha})$.
Now, this leads to the integration by substitution.
\end{proof}
\section*{Conclusion}
A type of generalized decomposition integral based on the Riemann-Lebesgue integral is introduced for
 real-valued functions with respect to two set functions. Several classical properties of this extension are discussed,
  including monotonicity with respect to the set, the function, and the set functions; homogeneity; 
  additivity concerning the set and the set functions; and a transformation rule.
\\
For future research, we plan to investigate additional properties of this generalized  integral, such as various inequalities, 
convergence results, and comparisons with other types of known integrals.

\section*{Acknowledgments}
{\scriptsize
This research has been accomplished within the
 UMI Group TAA - “Approximation Theory and Applications”,
 the G.N.AM.P.A. group of INDAM and the University of Perugia.
 The last author is a member of the ``Centro Interdipartimentale Lamberto Cesari'' of the University of Perugia.
The last author have been supported within the PRIN 2022 PNRR: 
”RETINA: REmote sensing daTa INversion with multivariate functional modeling
 for essential climAte variables characterization”, funded by the 
 European Union under the Italian National Recovery and Resilience Plan (NRRP) 
 of NextGenerationEU, under the Italian Ministry of University and Research
  (Project Code: P20229SH29, CUP: J53D23015950001).
}

\Addresses

\end{document}